\documentclass{article}
\usepackage{amsmath,amsthm,indentfirst}
\usepackage{amssymb}
\usepackage{amsfonts}
\usepackage{amscd}
\usepackage[latin1]{inputenc}

\theoremstyle{plain}
\newtheorem*{maintheorem*}{Main Theorem}
\newtheorem*{thm*}{Theorem}
\newtheorem*{thma*}{Theorem A}
\newtheorem*{thmaa*}{Theorem A'}
\newtheorem*{thmb*}{Theorem B}
\newtheorem*{thmo*}{Theorem 1.1}
\newtheorem*{thmc*}{Theorem C}
\newtheorem*{thmd*}{Theorem D}
\newtheorem*{thmf*}{Theorem 4.1}
\newtheorem*{conjecture*}{Conjecture}
\newtheorem*{prop*}{Proposition}
\newtheorem{thm}{Theorem}[section]
\newtheorem{cor}[thm]{Corollary}
\newtheorem{lem}[thm]{Lemma}
\newtheorem{prop}[thm]{Proposition}

\theoremstyle{definition}

\newtheorem*{proofc*}{Proof of Theorem C}

\newcommand{\lsup}[2]{%
        \ensuremath{{}^{#2}\!{#1}}}

\newtheorem{remark}[thm]{Remark}

\def\bbz{\mathbb{Z}}
\def\bbq{\mathbb{Q}}
\def\bbf{\mathbb{F}}
\def\bbr{\mathbb{R}}
\def\bba{\mathbb{A}}
\def\bbc{\mathbb{C}}
\def\bbh{\mathbb{H}}
\def\bbn{\mathbb{N}}
\def\bbg{\mathbb{G}}
\def\bbp{\mathbb{P}}
\def\bbt{\mathbb{T}}
\def\bbb{\mathbb{B}}
\def\bbu{\mathbb{U}}
\def\bbl{\mathbb{L}}
\def\bg{\bar{\mathbb{G}}}

\def\hfr{\mathfrak{h}}

\def\tfr{\mathfrak{t}}
\def\gl{\mathfrak{gl}}
\def\sl{\mathfrak{sl}}
\def\min{{\rm min}}
\def\Ad{{\rm Ad}}

\def\dim{{\rm dim\,}}
\def\char{{\rm char}}
\def\Hom{{\rm Hom}}
\def\vol{{\rm vol}}
\def\Lie{{\rm Lie}}
\def\diag{{\rm diag}}
\def\deg{{\rm deg}}
\def\comm{{\rm comm}}
\def\SL{{\rm SL}}
\def\o{{\cal O}}

\def\A{{\rm A}}
\def\B{{\rm B}}
\def\C{{\rm C}}
\def\D{{\rm D}}
\def\E{{\rm E}}
\def\F{{\rm F}}
\def\G{{\rm G}}

\title{Lattices of minimum covolume in Chevalley groups over local fields of positive characteristic  }
\author{Alireza Salehi Golsefidy}
\date{12/13/2007}
\begin{document}
\maketitle
\begin{abstract}
In this article, we show that  if $\bbg$ is a simply connected Chevalley group of either classical type of rank bigger than $1$ or type $\E_6$, and $q>9$ is a power of a prime number $p>5$,  then $G=\bbg(\bbf_q((t^{-1})))$, up to an automorphism, has a unique lattice of minimum covolume, which is $\bbg(\bbf_q[t])$.  
\end{abstract}
\section{Introduction and the statement of results}
  In the 40s, Siegel~\cite{Si} showed that  the (2,3,7)-triangle group is a lattice of minimum covolume in ${\rm SL}_2(\bbr)$. Later, Kazhdan and Margulis~\cite{KM} showed that if $G$ is a connected
  semisimple Lie group without compact factors, then there is a neighborhood of
  identity which does not intersect a conjugate of any given lattice. Combining
  this with a theorem of Chabauty~\cite{Ch}, one can get that there is a lattice which minimizes the covolume in any connected semisimple Lie group without compact factors. However from this proof the minimum covolume
  may not be determined directly. In the mid 80s, for $G={\rm SL}_2(\bbc)$,
  Meyerhoff~\cite{Me} showed that among the non-uniform lattices the minimum
  covolume is attained for ${\rm SL}_2(\o)$ where $\o$ is the ring of integers
  of $\bbq(\sqrt{-3})$. But there are co-compact lattices with smaller covolume,
  and the minimum has been just recently found by Martin and Gehring. In 1990, Lubotzky~\cite{Lu} studied the non-Archimedean analogue
  of the problem. He found the minimum covolume of lattices in $G={\rm SL}_2(\bbf_q((t^{-1})))$.
  Surprisingly, he showed that for odd primes the minimum occurs just for non-uniform
  lattices. Moreover, he showed that the minimum is attained for ${\rm SL}_2(\bbf_q[t])$.
   He also described fundamental domain of lattices in $G$ in the Bruhat-Tits
  tree. In particular, he proved that if $\Gamma$ is a lattice of minimum covolume, then
  its fundamental domain in the Bruhat-Tits tree should be a geodesic ray. In
  1999, Lubotzky and Weigel~\cite{LW} studied the characteristic zero case. They
  again used the action of lattices on the Bruhat-Tits tree. In this case by a classical
  result of Tamagawa all the lattices are co-compact. They proved that if the characteristic and the number of the elements of the residue field are large enough, then a lattice of minimum covolume  acts transitively on the vertices of the same color in the Bruhat-Tits tree. However in contrast to our intuition, there are lattices of minimum covolume which have more than one point of the same color in their fundamental domain.

  In this article  we will describe the possible structures of a lattice of minimum covolume in $G=\bbg(\bbf_q((t^{-1})))$ where $\bbg$ is a simply connected Chevalley group. Before stating the main results of this paper, let us fix a few notations. Let $F=\bbf_q((t^{-1}))$ and $\o=\bbf_q[[t^{-1}]]$. Let $\xi_l$ be the reduction map modulo $t^{-l}$ from $\bbg(\o)$ to $\bbg(\o/t^{-l}\o)$. Its kernel $G_l$ is called the $l^{th}$ congruence subgroup. Let $\Omega$ be the space of lattices in $G$, and $\Omega_l$ the subset of $\Omega$ consisting of lattices which intersect $G_l$ trivially. 
  
  \begin{maintheorem*}
  A lattice of minimum covolume in $G=\bbg(F)$, up to an automorphism of $G$, is $\bbg(\bbf_q[t])$ if $\bbg$ is a simply connected Chevalley group of either classical type of rank bigger than 1 or type $\E_6$, and $q$ is a power of a prime number $p>5$ which is larger than $9$ when $\bbg$ is not of type $\A$. 
  \end{maintheorem*}  
  
 To prove the main theorem, at the first step, we consider the action of the group of automorphisms of $G$ on $\Omega$. The result of Kazhdan-Margulis implies that there is a constant $\lambda$ depending just on $G$, such that any ${\rm Inn}(G)$-orbit in $\Omega$ intersects $\Omega_{\lambda}$. Our first result gives a quantitative version of this theorem.

  \begin{thma*}\label{KM2}
  Let $G=\bbg(\bbf_q((t^{-1})))$ where $\bbg$ is a simply connected Chevalley group.
  Let $l(\bbg)$ be the maximum coefficient of the simple roots in the highest root; then any $\Ad(\bbg)(F)$-orbit in $\Omega$ intersects $\Omega_{l(\bbg)}$. 
  \end{thma*}

 Note that for this theorem we do not have any restriction on the characteristic of $F$ or type of $\bbg$. This theorem provides us a fairly large ball in the fundamental domain of any lattice. Let us recall that $l(\bbg)$ is 1 only when $\bbg$ is of type A. So in the other cases, theorem A is not quite optimal. 
 
 \begin{thmb*}~\label{second}
 Let $\bbg$ be a simply connected Chevalley group of either classical type or type $\E_6$. Then an $\Ad(\bbg)(F)$-orbit in $\Omega$ consisting of lattices of minimum covolume intersects $\Omega_1$ if $q>3$ (resp. $q>2$) when $\bbg$ is (resp. is not) of type $\D_2$ or $\E_6$.
 \end{thmb*}
 
Altogether we can restrict ourselves to the lattices in $\Omega_1$. 

  \begin{thmc*}~\label{main}
   Let $\bbg$ be a simply connected  Chevalley group of rank bigger than one.  If  $p > 5$ (resp. $p>7$) when $\bbg$ is not (resp. is) of type $\G_2$, and $q>9$ when $\bbg$ is not of type $\A$, then a lattice of minimum covolume which is in $\Omega_1$, up to an automorphism of $G$, is $\bbg(\bbf_q[t])$.  \end{thmc*}
  
 \begin{remark}
 \begin{itemize}
 \item[i)] In theorem C, there is no restriction on the type of the Chevalley groups, and the type restriction in the main theorem is because of theorem B. As you will see the proof of theorem B is unified for the all the classical types and it is ad hoc for type $\E_6$.
 \item[ii)] Though in the proof of the main theorem, we use the higher rank assumption, to prove theorem A, we do not need any assumption on the rank. In particular, it is also valid for $\mathbb{SL}_2$. Using this fact and by the virtue of the proof of theorem C, we can also give an alternative proof for the work of Lubotzky (See~\cite[Chapter 5]{Sal}). 
 \item[iii)] Combining the description of the fundamental domain of a lattice of minimum covolume in $\SL_2(F)$, with the main theorem and using reduction theory, one can see that a lattice of minimum covolume  has a full Weyl chamber as a fundamental domain in the Bruhat-Tits building if $\bbg$ has a non-trivial center and is not of type $\E_7$, and  $p$ is large enough, e.g. $p>9$.
 \item[iv)] Lubotzky, in~\cite{Lu}, asked if a lattice of minimum covolume over a field of positive characteristic is ``generically" non-uniform. In a forthcoming paper~\cite{sa07}, I will partially give an affirmative answer to this question. I will prove the following theorem:
 \begin{thm} A lattice of minimum covolume  in $\bbg(\bbf_q((t)))$  is non-uniform if $\bbg$ is simply connected, absolutely almost simple over $\bbf_q((t))$ and the Tamagawa number over global function fields is 1.
 \end{thm}    
 \end{itemize}
 \end{remark}

\textbf{\textit{Structure of the paper.}} In the second section, we shall fix the notations used in this article, and recall a few known results which are used here. Section 3 is devoted to the proof of theorem A. In the section 4, using the result of Harder~\cite{Ha2} on the Tamagawa number of a simply connected Chevalley group over a global function field, we calculate the covolume of $\bbg(\o_k({\nu_0}))$ in $\bbg(k_{\nu_0})$, where $\nu_0$ is a place of degree 1. In particular, we give a formula for the covolume of $\bbg(\bbf_q[t])$ in $G=\bbg(\bbf_q((t^{-1})))$, and give a ``nice" upper bound for this value. We will also use Riemann's hypothesis for curves over finite fields to show that if the genus of $k$ is not zero, then the covolume of  $\bbg(\o_k({\nu_0}))$ in $\bbg(k_{\nu_0})$ is large. Theorem B is proved in the section 5. In the rest of the paper,  we are proving theorem C. Here is the outline of the proof of theorem C: 

  \begin{description}
  \item[Step 1.] ({\bf Large finite group}) Since by the assumption $\Gamma\cap G_1=\{1\}$, $\Gamma\cap\bbg(\o)$ can be embedded into $\bbg(\bbf_q)$ by $\xi_1$
  the reduction map modulo uniformizing element $\pi=t^{-1}$. It is not hard to see that $\xi_1(\Gamma\cap\bbg(\o))=\bbg(\bbf_q)$. In order to have a better description of  $\Gamma\cap\bbg(\o)$, we will show that  $H^1(SL_n(\bbf_q),\mathfrak{gl}_n(\bbf_q))=0$, and recall a result of J.~Bernstein ~\cite[corollary 6.3]{We} on the first cohomology of $\bbg(\bbf_q)$ with the adjoint action. Using these cohomology vanishing theorems, we prove that there is an adjoint automorphism $\theta$ such that $\theta(\Gamma)\cap\bbg(\o)=\bbg(\bbf_q)$.
  \item[Step 2.] ({\bf Arithmetic structure}) In the first step, we have seen that $\bbg(\bbf_q)$ can be assumed to be  a subgroup of $\Gamma$.
On the other hand,  by Margulis' arithmeticity~\cite[chapter IX]{Ma}, we know that $\Gamma$ has an arithmetic structure.  In this step, we will show that $\Gamma$ is commensurable to $\bbg(\o_k(\nu_0))$ where $k$ is a function field whose $\nu_0$-completion $k_{\nu_0}$ is isomorphic to $F$ (as a topological field) and $\o_k(\nu_0)$ is the ring of $\nu_0$-integers of $k$, i.e. $\{ x \in k \hspace{1mm}|\hspace{1mm} {\rm for}\hspace{1mm}{\rm any}\hspace{1mm} {\rm place}\hspace{1mm}\nu\neq\nu_0,\hspace{1mm}\nu(x)\geq 0\}$.

  \item[Step 3.] ({\bf Exact form of the lattice}) By the previous step and (the easy part of) the function field analogue of the Rohlfs' maximality criterion (See~\cite{Ro} or~\cite{CR}), proved as a proposition in a work by Borel and Prasad~\cite[Proposition 1.4]{BP}, we shall prove
  that $\Gamma$ is equal to the normalizer of $\bbg(\o_k(\nu_0))$ in $\bbg(k_{\nu_0})$ with the same notations as in step 2 and identifying $k_{\nu_0}$ by $F$.

  \item[Step 4.] ({\bf Possible $k$ and $\nu_0$}) It is an easy consequence of the previous steps to see that $\nu_0$ is of degree one. If $\bbg$ is centerless, then it is easy to show that the genus of $k$ is zero, and if $\bbg$ has a non-trivial center, we use Weierstrass' gap theorem (See~\cite{nt} or~\cite{A}) and get the same conclusion.  Then as $k$ is a genus zero global function field with a place of degree one, we conclude  that $k$ is isomorphic to the rational functions $\bbf_q(t)$.

  \item[Step 5.] ({\bf Maximality of $\bbg(\bbf_q[t])$}) In the final step using reduction theory~\cite{Ha1, Sp, P}, we  show that $N_{\bbg(F)}(\bbg(\bbf_q[t]))=\bbg(\bbf_q[t])$, which completes the proof of theorem C.
  \end{description}

\textbf{\textit{Acknowledgements.}} I would like to thank
Professor A.~Lubotzky for suggesting  this problem and his
constant encouragement. I am very grateful to Professor
G.~A.~Margulis for his useful advice and discussions without which
I could not solve this problem. I am also in debt to Professor 
G.~Prasad for carefully reading the earlier version of this aritcle 
and pointing out some errors. I would also like to thank the anonymous referee for extensive comments on an earlier version, which made the presentation of this article much more coherent and pleasant. 

\section{Notations and background.}

\textbf{\textit{Global function field.}} The algebraic and
separable closures of a field $k$ would be denoted by $\bar k$ and
$k_s$, respectively. Let $k$ be a global function field. The
subfield of $k$ consisting of algebraic elements over $\bbf_p$ is
called the \textit{constant field}. When we say $k/\bbf_q$ is a
global function field, we mean that its constant field is
$\bbf_q$. For any place $\nu$ of $k$, its completion with respect to $\nu$ 
and valuation ring are denoted by $k_{\nu}$ and $\o_{\nu}$, respectively. 
Ring of $\nu$-integers $\o_k(\nu)$ is equal to $\{ x \in k \hspace{1mm}|\hspace{1mm} {\rm for}\hspace{1mm}{\rm any}\hspace{1mm} {\rm place}\hspace{1mm}\nu'\neq\nu,\hspace{1mm}\nu'(x)\geq 0\}$.
Number of elements of the residue field of $k_{\nu}$ is
denoted by $q_{\nu}$. As a corollary of Riemann-Roch theorem, one
can see Weierstrass' gap theorem:
\begin{thm}~\cite{nt, A}
Let $k$ be a global function field of genus $g$, and $\nu$ a place of degree one. Then $\nu(\o_k(\nu))=\bbz^{\le 0}\setminus\{-a_1,\cdots,-a_g\}$, where $a_1, a_2, \cdots, a_g$ are $g$ natural numbers which are less than $2g$. 
\end{thm}
We just use the fact that such gap exists if $g$ is positive
i.e. -1 is not  in $\nu(\o_k(\nu))$. It is also well-known
that there are two types of genus zero function field. Such a field is
either a rational function field or function field of a conic. However a field of genus zero which has a place of degree one is isomorphic to a rational function field~\cite{A}.

Let $V_k$ be the set of all places of $k$. The zeta function $\zeta_k(s)$ of $k$ is defined as the following product $\zeta_k(s)=\prod_{\nu\in V_k}(1-q_{\nu}^{-s})^{-1}$.  A. Weil (e.g. see~\cite{Bom}) gave a complete description of the zeta function of $k$, proving ``Riemann hypothesis" for curves over finite fields.

\begin{thm}~\label{zeta}
Let $k/\bbf_q$ be a global function field of genus $g$. Then
 $$\zeta_k(s)=\frac{P_k(q^{-s})}{(1-q^{-s})(1-q^{1-s})},$$ where $P_k$ is a polynomial of degree $2g$ with integer coefficients, and its roots have absolute value $q^{-\frac{1}{2}}$.  
\end{thm}


\textbf{\textit{Unipotent subgroups of reductive groups in
characteristic $p>0$.}} Borel and Tits~\cite{BT} studied unipotent
subgroups of reductive groups and generalized some of the results
which were known in the characteristic zero. They associated a
$k$-closed parabolic $P$ to a $k$-closed unipotent subgroup $U$ of
a connected reductive $k$-group $\bbg$, such that
$N_{\bbg}(U)\subseteq P$ and $U\subseteq R_u(P)$. In general, even
if $U$ is defined over $k$, $P$ is not necessarily defined
over $k$. A unipotent subgroup $U$ of $\bbg(\bar{k})$ is called
$k$-embeddable if one can find a $k$-parabolic subgroup of $\bbg$
which contains $U$ in its unipotent radical. A unipotent element
$u$ of $\bbg(k)$ is called $k$-embeddable if the group generated
by $u$ is $k$-embeddable. Borel and Tits, in their fundamental
work, showed a number of important theorems some of which we recall
here.

\begin{thm}~\cite[proposition 3.1]{BT}~\label{ks-embeddable}
Let $\bbg$ be a connected reductive $k$-group and $U\subseteq
\bbg$ be a $k$-closed, $k_s$-embeddable unipotent subgroup of
$\bbg$. Then there exits a $k$-parabolic subgroup $\bbp$ of $\bbg$
such that:

\begin{itemize}
\item[(i)] $U\subseteq R_u(\bbp)$. \item[(ii)] $N_{\bbg}(U) \cap
\bbg(k_s) \subseteq \bbp$.
\end{itemize}
\noindent
 In particular, $U$ is $k$-embeddable.
\end{thm}

\begin{thm}~\cite[proposition 3.6]{BT}~\label{unipelem}
A unipotent subgroup of $\bbg(k)$ is $k$-embeddable if each of its
elements is $k_s$-embeddable.
\end{thm}

Later, J.~Tits~\cite{T} investigated all the pairs
$(\bbg,k)$ for which any $k$-closed unipotent subgroup is
$k$-embeddable. In view of theorems~\ref{ks-embeddable}
and~\ref{unipelem}, he treated elements of $\bbg(k)$ in
the $k$-split case. He proved that:

\begin{thm}~\cite[corollary 2.6]{T}~\label{good}
Let $\bbg$ be a quasi-simple connected $k$-group. Then, if $p$ is
not a torsion prime for $\bbg$, all unipotent elements of
$\bbg(k)$ are $k$-embeddable.
\end{thm}

Recently, P.~Gille~\cite{Gi} has completed the picture:
\begin{thm}~\cite[theorem 2]{Gi}~\label{unip}
Assume $\char(k)=p>0$, $[k:k^p]\le p$, and $\bbg$
is a semisimple simply connected $k$-group. Then every unipotent
subgroup of $\bbg(k)$ is $k$-embeddable.
\end{thm}

 Using theorems~\ref{ks-embeddable}, ~\ref{unip} and the
fact that for any local field $F$ of characteristic $p>0$,
$[F:F^p]=p$, we have:

\begin{thm}~\label{unipotent}
Let $F=\bbf_q((t^{-1}))$, $\bbg$ be a simple simply connected
$F$-group, and $U$ be a unipotent subgroup of $\bbg(F)$. Then
there is a $F$-parabolic subgroup $\bbp$ of $\bbg$ such that:

\begin{itemize}
\item[(i)] $U\subseteq R_u(\bbp)$. \item[(ii)] $N_{\bbg}(U) \cap
\bbg(F_s) \subseteq \bbp$.
\end{itemize}
\end{thm}

\textbf{\textit{Chevalley groups.}} Here we will recall some of
the well-known facts about absolutely almost simple algebraic
groups and set some of the notations which are used in this article.

Let $\bbg$ be an absolutely almost simple $k$-split $k$-group,
$\bbt$ a maximal $k$-split torus,
$\Phi=\Phi(\bbg,\bbt)$  the root system of $\bbg$ with respect to
$\bbt$, $\Delta$ a set of simple roots of $\Phi$, $\bbb$ the
Borel subgroup of $\bbg$ corresponded to $\Delta$,
$X(\bbt)=\Hom(\bbt,\bbg_m)$ the set of characters of $\bbt$,
and $P_r$  the subgroup of $X(\bbt)$ generated by the roots
$\Phi$. Killing form provides an inner product on the Euclidean
space $E=X(\bbt)\otimes_{\bbz}\bbr$. For any pair of roots
$\alpha$ and $\beta$, we denote the reflection with respect
to $\alpha$ by $\sigma_{\alpha}$, and we have
$\sigma_{\alpha}(\beta)=\beta-\langle\beta,\alpha\rangle\alpha$,
where
$\langle\beta,\alpha\rangle=2\frac{(\beta,\alpha)}{(\alpha,\alpha)}$
is an integer. Hence we have:
$$P_r\subseteq X(\bbt)\subseteq P=\{\theta\in
X(\bbt)\otimes_{\bbz}\bbq\hspace{.05in}|
\langle\theta,\alpha\rangle\in\bbz\hspace{.1in}\mbox{for any
root}\hspace{.05in} \alpha\}.$$ It is well-known that $\bbg$ is
simply connected (resp. adjoint) if and only if $X(\bbt)=P$
(resp.$X(\bbt)=P_r$) . So $\bbg$ is simply connected if and only
if $X(\bbt)$ has a base $\lambda_{\alpha}$ such that
$\langle\beta,\lambda_{\alpha}\rangle=\delta_{\beta,\alpha}$ for
any simple roots $\alpha$ and $\beta$, where
$\delta_{\beta,\alpha}$ is the Kronecker delta i.e. it is one if
$\alpha=\beta$ and zero otherwise. We may look at $\Ad(\bbg)$ as the group of inner
automorphisms of $\bbg$, and $\Ad(\bbg)(k)$ may be looked at as a
subgroup of the group of automorphisms of $\bbg(k)$. For any
$\chi\in\Hom(P_r,k^*)$ we get an automorphism of $\bbg(k)$ that we
shall denote by $h(\chi)$.

 By the classification of absolutely almost simple
groups, we know the type of the root system determines the group
up to isogeny. Each root system has a unique highest root. We
shall denote the biggest coefficient appearing in the highest root
of $\bbg$ by $l(\bbg)$ i.e. $l(\bbg)$ is 1 (resp. 2,2,2,3,4,6,4,3)
when $\bbg$ is of type A (resp. B, C, D, $\E_6$, $\E_7$, $\E_8$,
$\F_4$, and $\G_2$).

 To any subset $\Psi$ of $\Delta$, we associate a
$k$-parabolic $\bbp_{\Psi}$ which contains $\bbt$, and  any negative root appearing in whose
Lie algebra is a linear combination of elements of $\Phi\setminus\Psi$. We denote the opposite
parabolic of a given one $\bbp$ by $\bbp^-$.

For any root $\alpha$, one can find a unipotent
$k$-subgroup $\bbu_{\alpha}$ of $\bbg$ whose Lie algebra is the
$\alpha$-root subspace of $\Lie(\bbg)$. Moreover, there is an
$k$-isomorphism $u_{\alpha}$ between $\bbg_a$ and $\bbu_{\alpha}$.
\\

\textbf{\textit{Number of elements of $|\bbg(\bbf_q)|$.}} If
$\bbg$ is an algebraic group defined over a finite field $\bbf_q$,
one might wonder what the number of $\bbf_q$-rational
points of $\bbg$ is. For instance, by a theorem of Lang~\cite{La},
if $\bbg_1$ and $\bbg_2$ are two $\bbf_q$-isogenous groups, then
$|\bbg_1(\bbf_q)|=|\bbg_2(\bbf_q)|$.

 It is, if $\bbg$ is a Chevalley group, 
 well-known that $|\bbg(\bbf_q)|$ can be calculated by the
formula provided by the following theorem:

\begin{thm}~\cite[theorems 8.6.1,10.2.3]{Ca}~\label{|G(F)|} Let
$\bbg$ be a rank $r$ Chevalley group; then
$$|\bbg(\bbf_q)|=q^{{(\dim\bbg-r)}/2}\cdot\prod_{i=1}^r
(q^{m_i+1}-1),$$ where $m_1,\cdots,m_r$ are as follows:

$$\begin{array}{ll} & m_1,m_2,\cdots, m_r\\
\A_r & 1,2,\cdots,r\\
\B_r & 1,3,\cdots,2r-1\\
\C_r & 1,3,\cdots,2r-1\\
\D_r & 1,3,\cdots,2r-3,r-1\\
\G_2 & 1,5\\
\F_4 & 1,5,7,11\\
\E_6 & 1,4,5,7,8,11\\
\E_7 & 1,5,7,9,11,13,17\\
\E_8 & 1,7,11,13,17,19,23,29
\end{array}$$
\end{thm}

\textbf{\textit{Congruence subgroups.}} Let $F$ be a
nonarchimedean local field, $\o$ its valuation ring, $\pi$ a
uniformizing element, and $\bbf_q\simeq \o/\pi\o$ its residue
field. We shall denote the reduction
 map modulo $\pi^l$ from
$\bbg(\o)$ to $\bbg(\o/\pi^l\o)$ by $\xi_l$. Its kernel will be
called the $l^{th}$ congruence subgroup, and will be denoted by
$G_l$. Pre-image of $\bbb(\bbf_q)$ under $\xi_1$ is the stabilizer
of a chamber in the Bruhat-Tits building, such a subgroup is
called an Iwahori subgroup, and we shall denote it by $\cal{I}$.
By virtue of the following
easy proposition we can describe $h(\chi)(G_l)$ for a given $\chi\in\Hom(P_r,F^*)$. 
\begin{prop}~\label{cong}
Let $\bbg$ be a Chevalley group. Then the $l^{th}$ congruence
subgroup of $\bbg$ is topologically generated by $u_{\beta}(x)$ for $\beta\in\Phi$ and $x\in \pi^l\o$, and $T_l$ the $l^{th}$ congruence
 subgroup of $\bbt$.
\end{prop}
\begin{proof}
Let $H_l=\overline{\langle T_l, u_{\beta}(x) | \beta\in\Phi,
x\in\pi^l\o \rangle}$. Clearly
 $H_l \subseteq G_l$ and the image of $H_l$ under $\xi_{l+1}$ is the same as $\xi_{l+1}(G_l)$. Hence for any $l\geq1$,
 $H_lG_{l+1}=G_l$. By induction on $m$, we shall show that $H_lG_m=G_l$ for any $m\geq l+1$. We
 have shown it for $m=l+1$. By induction hypothesis, $H_lG_m=G_l$. On the other hand,
 $G_m=H_mG_{m+1}$ and $H_m \subseteq H_l$ for $m\geq l+1$.
 Thus $$G_l=H_lG_m=H_lH_mG_{m+1}=H_lG_{m+1},$$ and so
 $G_l=\bigcap_{l\leq m}H_lG_m=\overline{H_l}=H_l$.
\end{proof}

 Occasionally we will denote $\bbg(\o)$ by $G_0$. Though we
are not going to use it but it is worth mentioning that with an
appropriate metric on $G$ we may look at the congruence subgroups
as its balls centered at the identity.

\section{Proof of theorem A.}

For each lattice $\Gamma$, one can define a pair of numbers
$p_{\Gamma}=(n_{\Gamma},m_{\Gamma})$ where $n_{\Gamma}=\min\{l\in
\bbn|\Gamma\cap G_l=(1)\}$ and $m_{\Gamma}=|G_{n_{\Gamma}-1}\cap
\Gamma|$. Consider lexicographic order on
$\Sigma_{\Gamma}=\{p_{a(\Gamma)}|a\in \Ad(\bbg)(F)\}$.
Replacing $\Gamma$ with $a_0(\Gamma)$ for a suitable
$a_0\in \Ad(\bbg)(F)$, without lose of generality, one can
assume that $p_{\Gamma}$ is minimal in $\Sigma_{\Gamma}$. We will
show that $n_{\Gamma} \le l(\bbg)$. Assume, if possible, that 
$n_{\Gamma} > l(\bbg)$.

 Since $G_1$ is a pro-$p$ group, $\Gamma \cap G_1$ is a finite
$p$-group and so it is a unipotent subgroup. Hence by
theorem~\ref{unipotent}, there is a parabolic $\bbp$ defined over
$F$ such that $\Gamma \cap G_1\subseteq R_u(\bbp)(F)$ and
$N_G(\Gamma\cap G_1)\subseteq \bbp(F)$, and in particular $\Gamma
\cap G_0 \subseteq \bbp(F)$. By the structural theory of reductive
groups $\bbp$ is conjugate to a standard parabolic $\bbp_{\Xi}$.
Since $\bbg(F)=\bbg(\o)\bbp(F)$, $\bbp=\bbp_{\Xi}^x$ for some
$x\in G_0$ and $\Xi\subseteq \Delta$. It is clear that
$p_{\Gamma}=p_{\Gamma^{x^{-1}}}$, and so without loss of
generality we can assume that $\bbp=\bbp_{\Xi}$, and in particular
it is defined over $\bbf_p$. Now for any $\alpha \in \Xi$, we
define $\theta_{\alpha}=\theta \in \Hom(P_r,F^*)$ such that
$\theta(\alpha)=t^{-1}$ and $\theta(\alpha')=1$ for any
$\alpha'\in \Delta\setminus\{\alpha\}$. Now let us state three properties of $h(\theta)$:

\begin{itemize}
\item[{\bf (P1)}] $h(\theta)(u_{\beta}(x))=u_{\beta}(\theta(\beta)x)$ for
any $\beta \in \Phi$ and $x\in F$.
\item[{\bf (P2)}] Using (P1)
and the description of the congruence subgroups given in the
proposition~\ref{cong}, one has $h(\theta)(G_l) \subseteq
G_{l-l(\bbg)}$ for any $l\ge l(\bbg)$. 
\item[{\bf (P3)}] For any $l\ge
l(\bbg)$, the same argument as in (P2) gives us
$\xi_l(h(\theta)(G_l)) \subseteq
R_u(\bbp_{\alpha}^-)(\o/t^{-l}\o)$ and
$\xi_{l+1}(h(\theta)(G_l))\subseteq
\bbp_{\alpha}^-(\o/t^{-l-1}\o)$.
\end{itemize}

 Now, we consider $\Lambda=h(\theta)^{-1}(\Gamma)$, and divide the argument into several steps. 
\\

\noindent
\textbf{\textit{First Step:}} $n_{\Gamma}=n_{\Lambda}$.

\textit{Proof.} By the choice of $\Gamma$, we know that
$n_{\Lambda}\geq n_{\Gamma}$. So it is enough to show that
$n_{\Lambda}\le n_{\Gamma}=n$, i.e. $\Lambda\cap G_{n}=(1)$. Using
the contrary assumption and (P2),
we can see that $\Gamma\cap h(\theta)(G_n)$ is a subset of $\Gamma
\cap G_1$, and so $\Gamma \cap h(\theta)(G_n)$ is a subset of
$R_u(\bbp)(\o)$. However by (P3),
$\xi_{n}(h(\theta)(G_n))\subseteq
R_u(\bbp_{\alpha}^-)(\o/t^{-n}\o)$. Hence 
$\xi_{n}(\Gamma\cap h(\theta)(G_n))=1$, as we wished.
\\

\noindent
\textbf{\textit{Second Step:}} $\Gamma\cap G_{n-1}=\Gamma \cap
h(\theta)(G_{n-1})$.

\textit{Proof.} By the first step and the choice of $\Gamma$,
$|\Lambda\cap G_{n-1}|\ge |\Gamma\cap G_{n-1}|$. On the other
hand, by the assumption that $n>l(\bbg)$, and (P2), $\Gamma \cap h(\theta)(G_{n-1})$ is a subset of
$\Gamma \cap G_0$, and so it is contained in $\bbp(\o)$. Using (P3), one can conclude that
$\xi_{n-1}(h(\theta)(G_{n-1}))\subseteq
R_u(\bbp_{\alpha}^-)(\o/t^{-(n-1)}\o)$. Hence $\Gamma \cap
h(\theta)(G_{n-1})\subseteq \Gamma\cap G_{n-1}$. Therefore,
comparing the number of elements of them, we can see that the
equality holds, which completes the proof of this step.
\\

\noindent
\textbf{\textit{Final Step:}} By the previous step $\Gamma \cap
G_{n-1}=\Gamma \cap h(\theta_{\alpha})(G_{n-1})$, for any $\alpha
\in \Xi$. Thus by (P3),
$\xi_{n}(\Gamma\cap G_{n-1})\subseteq
\bbp_{\alpha}^-(\o/t^{-l}\o)$ for any $\alpha \in \Xi$. Since
$\bigcap_{\alpha\in\Xi}\bbp_{\alpha}^-=\bbp_{\Xi}^-=\bbp^-$ and
$\Gamma\cap G_{n-1}\subseteq R_u(\bbp)(\o)$, one can see that
$\Gamma$ should intersect $G_{n-1}$ trivially, which is a
contradiction.


\section{Covolume of $\bbg(\o_k({\nu_0}))$ in $\bbg(k_{\nu_0})$.}

In this section, we compute the covolume of $\bbg(\o_k({\nu_0}))$ in $\bbg(k_{\nu_0})$, where $\bbg$ is as always a simply connected Chevalley group, $k/\bbf_q$ is a global function field, and $\nu_0$ is a place of degree 1. Following Prasad's~\cite{Pr} treatment, we start with a result of Harder~\cite{Ha2} on the Tamagawa number of Chevalley groups over the global function fields, and then use strong approximation to compute the considered quantity. We should say that one can get the desired result directly from Prasad's formula. However as we are working with split groups, we decided to include the argument here.

 Since $\bbg$ is a simply connected Chevalley group, for any place $\nu$ of $k$, $P_{\nu}=\bbg(\o_{\nu})$ is a hyperspecial parahoric subgroup of $\bbg(k_{\nu})$ with smooth reduction map modulo the uniformizer $\pi_{\nu}$ to the finite group $\bbg(\bbf_{q_{\nu}})$. Hence $\omega^*_{\nu}(P_{\nu})=|\bbg(\bbf_{q_{\nu}})|/ q_{\nu}^{\dim \bbg}$. On the other hand, strong approximation implies that $\bbg(\bba_k)= \bbg(k_{\nu_0})\prod_{\nu\neq\nu_0}P_{\nu}\cdot\bbg(k)$, where $\bbg(k)$ is embedded diagonally in $\bbg(\bba_k)$. Therefore there is a fibration  $\bbg(\bba_k)/\bbg(k)\rightarrow\bbg(k_{\nu_0})/\bbg(\o_k(\nu_0))$ with fiber $\prod_{\nu\neq\nu_0}P_{\nu}$. Hence $$\omega^*_{\bba_k}(\bbg(\bba_k)/\bbg(k))=\omega^*_{\nu_0}(\bbg(k_{\nu_0})/\bbg(\o_k(\nu_0)))\cdot\prod_{\nu\neq\nu_0}\frac{|\bbg(q_{\nu})|}{q_{\nu}^{\dim\bbg}}.$$

On the other hand, by Harder's result $\omega^*_{\bba_k}(\bbg(\bba_k)/\bbg(k))=q_k^{(g_k-1)\dim\bbg}$. We also would like to normalize the Haar measure $\mu$ such that $\mu(G_1)=1$. So we should multiply everything by $q_{\nu_0}^{\dim\bbg}$, as $\omega^*_{\nu_0}(G_1)=1/ q_{\nu_0}^{\dim\bbg}$. Therefore, using theorems~\ref{zeta} and~\ref{|G(F)|}, and $q_{\nu_0}=q$, we have:
$$\begin{array}{ll}
\mu(\bbg(k_{\nu_0})/\bbg(\o_k({\nu_0})))&=q^{(g_k-1)\dim\bbg}|\bbg(\bbf_{q_{\nu_0}})|\cdot\prod_{\nu\in V_k}q_{\nu}^{\dim\bbg}/|\bbg(\bbf_{q_{\nu}})|\\ &\\
&= q^{g_k\dim\bbg}\cdot\prod_{i=1}^r(1-q^{-m_i-1})\cdot\prod_{\nu}\prod_{i=1}^r(1-q_{\nu}^{-m_i-1})^{-1}\\ &\\&=q^{g_k\dim\bbg}\cdot\prod_{i=1}^r(1-q^{-m_i-1})\cdot\prod_{i=1}^r \zeta_k(m_i+1)\\ &\\
&=q^{g_k\dim\bbg}\cdot\prod_{i=1}^rP_k(q^{-m_i-1})\cdot\prod_{i=1}^r(1-q^{-m_i})^{-1}.
\end{array}$$ 
In particular, for  $k=\bbf_q(t)$ and $\nu_0$ its ``infinite" place, we have $$\mu(\bbg(\bbf_q((t^{-1})))/\bbg(\bbf_q[t]))=\prod_{i=1}^r(1-q^{-m_i})^{-1}.$$

On the other hand, for any number $s\ge1$, by theorem~\ref{zeta}, it is easy to see that, $$q^g\cdot P_k(q^{-s})\ge (1-q^{-1/2})^{2g}.$$ 

Thus if $k/\bbf_q$ is a global function field of genus $g$, $\nu_0$ is a place of degree one,  then 
$$
\mu(\bbg(k_{\nu_0})/\bbg(\o_k({\nu_0})))\ge (q-q^{1/2})^{2g r}.$$

\begin{cor}~\label{genus}
Let $k/\bbf_q$ be a global function field of genus $g>0$, $\bbg$ a simply connected Chevalley group of rank larger than 1. If $q>2$, then $$
\mu(\bbg(k_{\nu_0})/\bbg(\o_k({\nu_0})))>2.$$ 
\end{cor}

We also would like to have a numerical bound for the minimum covolume
of lattices in $G=\bbg(\bbf_q((t^{-1})))$. So we will show an easy
inequality which provides us the upper bound 2 for the minimum covolume of lattices in
$G$.

\begin{cor}\label{upperbound}
Let  $\bbg$ be a simply connected Chevalley group. Assume that $q\geq 4$ or $q\geq
3$ whenever $\bbg=\mathbb S{\rm pin}_4$ or not, respectively. Then
$$\vol(\bbg(\bbf_q((t^{-1})))/\bbg(\bbf_q[t]))<2.$$
\end{cor}
\begin{proof}
For $\bbg=\mathbb S{\rm pin}_4$, it is easy to get the inequality
by the above calculation. For the other cases, again using
the above formula, it is not hard to see that it would be
enough to show that $F(x)=\prod_{i=1}^{\infty} (1-x^{-i})^{-1}$ is
convergent and less than two, for $x\geq 3$. Let
$F_n(x)=\prod_{i=1}^{n}(1-x^{-i})^{-1}$.
$$\begin{array}{lll}
\ln(F_n(x))&=&-\sum_{i=1}^n \ln (1-x^{-i})=\sum_{i=1}^n\sum_{j=1}^{\infty} \frac {x^{-ij}} {j}\\\\

&=&\sum_{j=1}^{\infty}\sum_{i=1}^{n} \frac {x^{-ij}} {j}<\sum_{j=1}^{\infty}\sum_{i=1}^{\infty} \frac {x^{-ij}} {j}\\\\
&=&\sum_{j=1}^{\infty}\frac{1}{j(x^j-1)}\leq \sum_{j=1}^{\infty}\frac{1}{j(3^j-1)}\\\\
&=&\frac{1}{2}+\frac{1}{16}+\sum_{j=3}^{\infty}\frac{1}{j(3^j-1)}\\\\
&<&\frac{1}{16}-\frac{1}{8}+\frac{1}{2}+\frac{1}{8}+\sum_{j=3}^{\infty}\frac{1}{j\cdot 2^j}\\\\
&=&-\frac{1}{16}+\sum_{j=1}^{\infty}\frac{1}{j\cdot
2^j}=\ln(2)-\frac{1}{16}.
\end{array}$$

\noindent

Hence $\{F_n(x)\}_{n=1}^{\infty}$ is increasing and bounded by
two. Thus it is convergent and $F(x)<2$.
\end{proof}

\section{Proof of theorem B.}
Let $\Gamma$ be a lattice of minimum covolume. So, by corollary~\ref{upperbound},  its covolume in $G=\bbg(F)$ is smaller than 2. Moreover, by theorem A, we may and will assume that $\Gamma$ intersects $G_{l(\bbg)}$ trivially. We start with the classical types and then take care of groups of type $\E_6$.
\\

\textit {\textbf{Classical types.}}
Since $l(\bbg)=1$ when $\bbg$ is of type A, we may and will assume that $\bbg$ is not of type A. For the other classical types, $l(\bbg)=2$, so  $\Gamma\cap G_2=(1)$. $\Gamma \cap G_1$ is a unipotent subgroup. If it is trivial, there is nothing to prove, otherwise by
theorem~\ref{unipotent} and changing $\Gamma$ by a similar argument as we have seen in the proof of theorem A, one can find a standard parabolic $\bbp=\bbp_{\Xi}$, such that $\Gamma \cap G_1
\subseteq R_u(\bbp)(\o)$ and $\Gamma \cap G_0\subseteq \bbp(\o)$. On the other hand, since $\bbp\cap R_u(\bbp^-)=(1)$, it is clear that $|\bbg(\bbf_q)|\geq |\bbp(\bbf_q)|\cdot |R_u(\bbp)(\bbf_q)|$.
Therefore as a direct corollary, we have $|\Gamma \cap G_0| \le |\bbg(\bbf_q)|$. In fact, if $R_u(\bbp)$ is not abelian, one can get $|\Gamma \cap G_0| \le p^{-1}|\bbg(\bbf_q)|$ since $\Gamma\cap
G_1$ is abelian. If $R_u(\bbp)$ is abelian, $\bbp$ should be a maximal parabolic $\bbp_{\alpha}$ and the coefficient of $\alpha$ in the highest root should be 1. Hence by the similar argument as
we had in the proof of theorem A, $h(\theta_{\alpha})^{-1}(\Gamma)$ intersects the first congruence
subgroup trivially. So, we may assume that $|\Gamma \cap G_0| \le p^{-1}|\bbg(\bbf_q)|$.

However, we claim that under this assumption the covolume of $\Gamma$
is bigger than 2.  In order to show this claim, it is enough to note the following trivial lemma:

\begin{lem}~\label{index-vol} With the previous notations:
$$\vol(G/{\Gamma})\geq\frac{|\bbg(\bbf_q)|}{|\Gamma\cap G_0|}.$$
\end{lem}
\begin{proof} It is well-known that,
$\vol(G/{\Gamma})=\mu(G_0)\cdot\sum_{x\in
T}\frac{1}{|G_0\cap x\Gamma x^{-1}|}$ where $T$ is a set of
representatives of the double cosets of $G_0$ and $\Gamma$ in $G$.
In particular, $$\vol(G/{\Gamma})\geq\frac{|G_0/G_1|}{|\Gamma\cap G_0|}=\frac{|\bbg(\bbf_q)|}{|\Gamma\cap G_0|}.$$ 
\end{proof}

\noindent By the above lemma and  our assumption on $\Gamma$, we have  $\vol(G/\Gamma)\geq p\geq 2$, which is a contradiction. This finishes the proof of theorem B, for groups of classical types.
\\

\textit{\textbf{Type $\E_6$.}} 
Since $\bbg$ is of type $\E_6$,  $l(\bbg)=3$, and so $\Gamma\cap G_3=(1)$. Again since $\Gamma$ is a lattice of minimum covolume, by corollary~\ref{upperbound}, its covolume is less than 2. Now under the restriction on the covolume, we shall push $\Gamma$  out of $G_2$. As soon as, we get $\Gamma\cap G_2=1$, by a similar argument as the classical case, we can also push $\Gamma$ out of the first congruence subgroup. 

\noindent
\textit{\textbf{Pushing out of $G_2$:}} As in the proof of theorem A, after replacing  $\Gamma$ with $h(\Gamma)$ for some $h\in \Ad(\bbg)$, we can assume that  $\Gamma\cap G_3=1$ and $|\Gamma\cap G_2|$ is minimum among such choices.  If $\Gamma\cap G_2$ is trivial, we are done. If not, with a similar argument as in the proof of theorem A,  we can further assume that, $\Gamma\cap G_0\subseteq \bbp$, and $\Gamma\cap G_1\subseteq R_u(\bbp)$, where $\bbp=\bbp_{\Xi}$ is a standard parabolic.
For any $\alpha\in\Xi$ whose coefficient in the highest root is less than $l(\bbg)=3$, we can apply the same line of argument as in the proof of theorem A, and get that:
\begin{itemize}
\item[i)] $h(\theta_{\alpha})^{-1}(\Gamma)\cap G_3=1$.
\item[ii)] $\Gamma\cap h(\theta_{\alpha})(G_2)=\Gamma\cap G_2$.
\end{itemize} 

In the root system of type $E_6$, only one simple root $\alpha_0$ has coefficient $l(\bbg)=3$ in the highest root. Hence $\xi_3(\Gamma\cap G_2)\subseteq \bbp^-_{\Xi\setminus\{\alpha_0\}}$. So if $\alpha_0$ is not in $\Xi$, since $R_u(\bbp)\cap \bbp^-=1$, we are done. So we can assume that $\alpha_0\in \Xi$. Then $\xi_3(\Gamma\cap G_2)\subseteq \mathbb{U}_{\alpha_0}(\o/t^{-3}\o)$. $H=(\Gamma\cap G_0)/(\Gamma\cap G_1)$ can be realized as a subgroup of $\bbp(\bbf_q)$, also $\Gamma\cap G_2$ can be realized as a subgroup of $\Lie(R_u(\bbp))(\bbf_q)$, and since $\Gamma\cap G_2$ is a normal subgroup which commutes with $\Gamma\cap G_1$,  $H$ acts on it by conjugation, this action can be viewed as the action induced by the adjoint action of $\bbp$ on $\Lie(R_u(\bbp))$. By the above argument, the subgroup corresponded to $\Gamma\cap G_2$ in $\Lie(R_u(\bbp))(\bbf_q)$ is a non-trivial subgroup $U$ of $\Lie(\mathbb{U}_{\alpha_0})(\bbf_q)$.  The following lemma implies that though $U$ is finite, it is ``rigid" enough to give us some information at the level of algebraic groups.

\begin{lem}~\label{connec} $N_{\bbg}(U)=\{g\in\bbg\hspace{1mm}|\hspace{1mm} \Ad(g)(U)=U\}$ is a subgroup of a connected $\bbf_q$-algebraic group $\bbl$, with Lie algebra equal to $\mathfrak{l}=\Lie(\bbt)\oplus\sum_{\beta\in\Phi_{\alpha_0}}\Lie(\bbu_{\beta}),$ where $\Phi_{\alpha_0}=\{\phi\in\Phi\hspace{1mm}|\hspace{1mm}\alpha_0+\beta\not\in\Phi\}.$ 
\end{lem}
\begin{proof}
It is easy to see that $N_{\bbp}(U)\subseteq N_{\bbp}(\bbu_{\alpha_0})$. We want to show that $\bbl=N_{\bbg}(\bbu_{\alpha_0})$ has the claimed properties. Since $E_6$ is simply laced, Weyl group acts transitively on the roots, so $\alpha_0$ can be sent to the highest root. Then clearly the normalizer contains the minimal parabolic, and therefore itself is a parabolic and so is connected. The statement about the Lie algebra is clear. 
\end{proof}

\noindent
Now, let us introduce a few notations:
\begin{itemize}
\item[i)] $\langle\Xi\rangle$ be the subset of positive roots in whose linear combinations in terms of simple roots at least one of the elements of $\Xi$ has a non-zero coefficient. 
\item[ii)] $n_{\Xi,\alpha_0}:=|\Phi_{\alpha_0}\setminus -\langle\Xi\rangle|$.
\item[iii)] $m_{\Xi,\alpha_0}:=|\Phi_{\alpha_0}\cap\langle\Xi\rangle|$.
\end{itemize}

Now we claim that $\dim (\bbl)+\dim (\bbl\cap R_u(\bbp))+1<\dim(\bbg)$.  By the above lemma, $\dim(\bbl)\le r+n_{\Xi,\alpha}$, and $\dim (\bbl\cap\Lie(R_u(\bbp)))\le m_{\Xi,\alpha}$, where $r=6$ is the rank. Since $\dim \bbg=r+|\Phi|$,  we need to show that $n_{\Xi,\alpha_0}+m_{\Xi,\alpha_0}+1<|\Phi|$. To show this we need to recall a few facts about simply laced root systems:
\begin{remark} Let $\alpha,\beta\in\Phi$ then:
\begin{itemize}
\item[i)]  $\alpha+\beta\in\Phi$ if and only if $\langle\alpha,\beta\rangle=-1.$
\item[ii)] $\langle\alpha,\beta\rangle=-1$ if and only if $\langle\alpha,\alpha+\beta\rangle=1.$
\end{itemize}
\end{remark}

\noindent
Using the above remark, we have: 
$$|\{\beta\in\langle\Xi\rangle\hspace{1mm}|\hspace{1mm}\langle\beta,\alpha_0\rangle=-1\}|=
|\{\beta\in\langle\Xi\rangle\hspace{1mm}|\hspace{1mm}\langle\beta,\alpha_0\rangle=1\}|.$$
Therefore $n_{\Xi,\alpha_0}+m_{\Xi,\alpha_0}+1=|\{\beta\in\Phi\hspace{1mm}|\hspace{1mm}\langle\alpha_0,\beta\rangle\neq -1\}|$, which is clearly less than $|\Phi|$.

\noindent
On the other hand,  $|\Gamma\cap G_2|\le q$, and  the following two inequalities are clear: $$|(\Gamma\cap G_0)/(\Gamma\cap G_1)|\le|\bbl(\bbf_q)|\hspace{1mm}\& \hspace{1mm} |(\Gamma\cap G_1)/(\Gamma\cap G_2)|\le|\bbl(\bbf_q)\cap R_u(\bbp)(\bbf_q)|.$$ It is not hard to check that $|\bbl(\bbf_q)|\le q^{\dim\bbl}$, as $\bbl$ is a connected algebraic group by lemma~\ref{connec}. Therefore by the virtue of corollary~\ref{upperbound}, for $q\ge 3$,
$$|\Gamma\cap G_0|\le q^{\dim\bbg-1}\le 2q^{-1}|\bbg(\bbf_q)|.$$
Now by lemma~\ref{index-vol}, we have $\vol(G/\Gamma)\ge \frac{|\bbg(\bbf_q)|}{|\Gamma\cap G_0|}\ge q/2.$ So for $q\ge 4$, we would get a contradiction, which finishes the proof.


\section{Remarks on cohomology of finite groups.}
Using results from~\cite{St, CPS1, CPS2}, J.~Bernstein showed:

\begin{thm}~\label{cohom1}
If $\bbg$ is an absolutely almost simple simply connected
algebraic group over $\bbf_q$ and $q>9$, then
$$H^1(\bbg(\bbf_q),\Lie(\bbg)(\bbf_q))\simeq Z(\bbf_q)^*,$$
where $Z$ is equal to the center of $\Lie(\bbg)$ and
$Z(\bbf_q)^*=\Hom_{\bbf_q}(Z(\bbf_q),\bbf_q)$.
\end{thm}

So as a corollary, for $p>3$, $q>9$ and a simply connected
Chevalley group $\bbg$ not of type A, $H^1(\bbg(\bbf_q),
\Ad)$ is trivial. Type A is taken care of in the following
theorem. Our proof is elementary and different from J.~Bernstein's.

\begin{thm}~\label{cohom2}
$H^1(SL_n(\bbf_q),\gl _n(\bbf_q))=0$, where $p,n>2$ and
$SL_n(\bbf_q)$ acts by conjugation on $\gl _n(\bbf_q)$.
\end{thm}
\begin{proof}
Let $\delta$ be a 1-cocycle from $SL_n(\bbf_q)$ to
$\gl_n(\bbf_q)$. Since $|\bbt(\bbf_q)|$ and $p$ are coprime, we
may define $x=|\bbt(\bbf_q)|^{-1}\sum_{t\in\bbt(\bbf_q)}\delta
(t)$.
 Then for any $t'\in \bbt(\bbf_q)$,
 we have $\delta (t')=x-\lsup{x}{t'}$ since:

$$x=|\bbt(\bbf_q)|^{-1} \cdot \sum_{t\in\bbt(\bbf_q)}\delta (t't)
=|\bbt(\bbf_q)|^{-1} \cdot \sum_{t\in\bbt(\bbf_q)}(\lsup{\delta
(t)}{t'} +\delta (t')) =\lsup{x}{t'}+\delta (t').
$$

So without loss of generality, we may assume that the restriction
of $\delta$ to $\bbt(\bbf_q)$ is zero. Thus for any $g\in
\bbg(\bbf_q)$ and $t\in\bbt(\bbf_q)$, one has
$\delta(tg)=\lsup{\delta(g)}{t}$ and $\delta(gt)=\delta(g)$. Now
for any $1\leq i\leq n-1$, let the simple root
$\alpha_i(\diag(t_1,\cdots,t_n))=t_it_{i+1}^{-1}$, and so
$u_{\alpha_i}(x)=I_n+e_{i,i+1}(x)$ for any $x\in\bbf_q$. Since
$n>2$ it is clear that for any $1\leq i\leq n-1$ and $x\in\bbf_q$,
one can find $t\in\bbt(\bbf_q)$ such that $\alpha_i(t)=x$. Now by
the above discussion we know that $\delta(u_{\alpha_i}(1))$
commutes with all the elements $t\in\bbt(\bbf_q)$ which are in the
kernel of $\alpha_i$. On the other hand, since $p$ is odd, as we
mentioned above, one can find $t\in\bbt(\bbf_q)$ such that
$\alpha_i(t)=2$, and so we get:
$$t\delta(u_{\alpha_i}(1))t^{-1}=\delta(u_{\alpha_i}(1)^2)=
\delta(u_{\alpha_i}(1))+\lsup{\delta(u_{\alpha_i}(1))}{u_{\alpha_i}(1)}.$$

Therefore, it is easy to see that,  there are constants
$x_i\in\bbf_q$, for $1\le i\le n-1$, such that
$\delta(u_{\alpha_i}(x))=e_{i,i+1}(x_ix)$, for any $x\in\bbf_q$.
Clearly, there is a diagonal element
$y=\diag(y_1,\cdots,y_n)\in\gl_n(\bbf_q)$ (not necessarily in
$\sl_n(\bbf_q)$) for which we have
$d\alpha_i(y)=y_i-y_{i+1}=x_i$. Thus for any $i$, $x\in\bbf_q$,
and $t\in\bbt(\bbf_q)$, we have:

$$\delta(u_{\alpha_i}(x))=y-\lsup{y}{u_{\alpha_i}(x)}\hspace{.1in}\&\hspace{.1in}
\delta(t)=0=y-\lsup{y}{t}. $$

Therefore we may assume that the restriction of $\delta$ to the
upper triangular matrices  is zero since $\{g\in SL_n(\bbf_q)|
\delta(g)=0\}$ is a subgroup of $SL_n(\bbf_q)$.

 With the similar argument as above, for any $0\leq i\leq n-1$,
there are $x_i\in \bbf_q$, such that
$\delta(u_{-\alpha_i}(x))=e_{i+1,i}(x_ix)$, where
$u_{-\alpha_i}(x)=I_n+e_{i+1,i}(x)$. At the end, it is enough
to note that for any $i$:
$$u_{\alpha_i}(1)u_{-\alpha_i}(1)u_{\alpha_i}(1)^{-1}=
u_{-\alpha_i}(1)u_{\alpha_i}(1)^{-1}u_{-\alpha_i}(1)^{-1},$$

\noindent and then applying $\delta$ to the both sides, we find
out that for any $i$, $x_i=0$, which finishes the proof.
\end{proof}

\begin{remark} Using the above proof, it is easy to get the special
case of the theorem~\ref{cohom1} for $\mathbb{SL}_n$.
\end{remark}

We will use theorems~\ref{cohom1} and ~\ref{cohom2} to prove:

\begin{thm}~\label{finiteG0}
Let $q>9$, $p>3$, $\o=\bbf_q[[t^{-1}]]$, $\bbg$ be a simply
connected Chevalley group of rank at least 2, and $H$ be
a finite subgroup of $\bbg(\o)$. Assume that $\xi_1$ induces an
isomorphism $\xi$ between $H$ and $\bbg(\bbf_q)$; then:

\begin{itemize}
\item[(i)] If $\bbg$ is not of type $\A$, then there is
$g\in\bbg(\o)$, such that $gHg^{-1}=\bbg(\bbf_q)$. \item[(ii)] If
$\bbg$ is of type $\A_{n-1}$, then there is $g\in GL_n(\o)$ such
that $gHg^{-1}=\bbg(\bbf_q)$.
\end{itemize}
\end{thm}
\begin{proof}
Let $\pi=t^{-1}$, and $\bg$ be $\mathbb{GL}_n$ if $\bbg$ is of
type A and $\bbg$ itself otherwise. By induction on $m$, we will
find $g_m$ in the first congruence subgroup $\bar G_1$ of
$\bg(\o)$ such that $g_mgg_m^{-1}\equiv\xi(g) \pmod{\pi^m}$. After
proving this, because of the compactness of $\bar G_1$, a
subsequence of $\{g_m\}_{m=1}^{\infty}$ would converge to an
element $x\in \bar G_1$, and clearly $x$ provides us the claim.

For $m=1$, there is nothing to prove since by our assumption
$g_1=1$ works. Assume that we have found $g_m$ by the above
property, and let us find $g_{m+1}$. For any $g \in \bbg(\bbf_q)$,
by the definition of $g_m$, there is
$\delta(g)\in\Lie(\bg)(\bbf_q)$ for which we have:
$$g_m\xi^{-1}(g)g_m^{-1}=g+\pi^m\delta(g)g+\cdots\, .$$

One can easily check that $\delta$ is a 1-cocycle from
$\bbg(\bbf_q)$ to $\Lie(\bg)(\bbf_q)$. By theorems~\ref{cohom1}
and~\ref{cohom2}, there is $\bar x\in\Lie(\bg)(\bbf_q)$ such that
$\delta(g)=g\bar xg^{-1}-\bar x$. On the other hand, there is
$x\in \bar G_m$ such that $x\equiv 1+\pi^m \bar x
\pmod{\pi^{m+1}}$. Hence it can be easily checked that
$g_{m+1}=xg_m$ satisfies our claim.
\end{proof}

Let us return to the proof theorem C. Here we would like to show that if $\Gamma$ is a lattice of minimum covolume and in $\Omega_1$, after changing $\Gamma$ with an automorphism of adjoint type, we can assume that $\Gamma\cap G_0=\bbg(\bbf_q)$. Since $\Gamma$ is in $\Omega_1$, $\xi_1$ embeds $\Gamma\cap G_0$ into
$\bbg(\bbf_q)$. By lemma~\ref{index-vol}, we also know that
$\vol(G/\Gamma)\geq \frac{|\bbg(\bbf_q)|}{|\Gamma\cap G_0|}$. On the other hand, since $\Gamma$ is a lattice of minimum covolume, its covolume is at most the covolume of $\bbg(\bbf_q[t])$, and so by corollary~\ref{upperbound} its covolume is less than 2.
Hence $\xi_1$ induces an isomorphism between $\Gamma\cap G_0$ and $\bbg(\bbf_q)$.
Now, let us  use theorem~\ref{finiteG0}, for $H=\Gamma\cap G_0$. Thus, as $\bg(\o)$ normalizes the congruence subgroups of $\bbg$, we may assume that $\Gamma\cap G_0=\bbg(\bbf_q)$ after changing $\Gamma$ by an adjoint automorphism, as we wished.

\section{Arithmetic structure of $\Gamma$.}

By Margulis' arithmeticity, since $\bbg$ is a simply connected
absolutely almost simple group, there are a function field $k$
whose $\nu_0$-completion for a place $\nu_0$ is topologically
isomorphic to $F=\bbf_q((t^{-1}))$, a connected absolutely almost
simple $k$-group $\bbh$, and a $k_{\nu_0}$-isomorphism $\phi$ from
$\bbg$ to $\bbh$, such that $\phi(\Gamma)$ is commensurable with
$\bbh(\o_k(\nu_0))$ after identifying $k_{\nu_0}$ by $F$. It is worth mentioning that  $\bbh(\o_k(\nu_0))$ is well-defined only up to commensurability. Let
$\phi(\bbt)=\bbt'$. Clearly $\bbt'$ is a maximal torus and $\phi$
induces a bijection between $\Phi(\bbg,\bbt)$ and
$\Phi'=\Phi(\bbh,\bbt')$, so occasionally we may refer to $\Phi'$
by $\Phi$ according to this bijection. We denote the weight spaces
of $\bbt'$ by $\hfr_{\alpha'}$, so we have:
$$\Lie(\bbh)=\Lie(\bbt')\oplus\bigoplus_{\alpha'\in\Phi'}\hfr_{\alpha'}.$$
 One knows that
  $\Ad(\phi(\Gamma))\subseteq\comm(\Ad(\bbh(\o_k(\nu_0))))=\Ad(\bbh)(k)$
 since $\phi(\Gamma)$ and $\bbh(\o_k(\nu_0))$ are commensurable~\cite[see chapter VII(6.2)]{Ma}.
In particular, by the first step $\phi(\bbg(\bbf_q))$ preserves
the $k$-structure $\hfr_k$ of $\Lie(\bbh)$. Note that $\Lie(\bbh)$ is a vector space over $\mathcal{F}$ a ``universal domain" i.e. an algebraically closed field with uncountable transcendental basis over $\bbf_p$, and $\hfr_k$ is a Lie algebra over $k$ such that $\Lie(\bbh)=\hfr_k\otimes_k \mathcal{F}$. We claim that,
\begin{equation}\label{kstructure}\hfr_k=(\hfr_k\cap\Lie(\bbt'))\oplus\bigoplus_{\alpha'\in\Phi'}(\hfr_k\cap\hfr_{\alpha'})\end{equation}
Let $h\in\hfr_k$. We know that there are $x\in\Lie(\bbt')$ and
$x_{\alpha'}\in\hfr_{\alpha'}$ such that
$h=x+\sum_{\alpha'\in\Phi'}x_{\alpha'}$. We will show that $x$ and $x_{\alpha'}$'s are in $\hfr_k$, which implies equation (\ref{kstructure}). By the above discussion,
for any $t\in\bbt(\bbf_q)$,

$$\Ad(\phi(t))(h)=x+\sum_{\alpha'\in\Phi'}\alpha'(\phi(t))x_{\alpha'}
               =x+\sum_{\alpha\in\Phi}\alpha(t)x_{\alpha'}$$
is in $\hfr_k$. So if $\{(1,(\alpha(t))_{\alpha\in\Phi})\in
\bbf_q^{|\Phi|+1}| t\in\bbt(\bbf_q)\}$ spans the whole space
$\bbf_q^{|\Phi|+1}$, then one can easily write $x$ or $x_{\alpha'}$'s as linear combination of $\Ad(\phi(t))(h)\in \hfr_k$ over $\bbf_q$, which completes the proof of our claim. 
\begin{lem}~\label{poly}
Let $Q(T_1,\cdots,T_n)\in\bbf_q[T_1,\cdots,T_n]$. If $\deg_{T_i}
Q\leq q-2$ for any $i$  and $Q(a_1,\cdots, a_n)=0$ when $a_1
a_2\cdots a_n\neq 0$, then $Q=0$.
\end{lem}
\begin{proof} It can be easily proved by induction on the number of variables.
\end{proof} 
\begin{lem}~\label{root}
Let $q>5$ (resp. $q>7$) when $\bbg$ is not  (resp. is) of type $\G_2$; Then $\{(1,(\beta(t))_{\beta\in\Phi})\in
\bbf_q^{|\Phi|+1}| t\in\bbt(\bbf_q)\}$ spans the whole space
$\bbf_q^{|\Phi|+1}$.
\end{lem}
\begin{proof}
We proceed by contradiction. If it spans a proper subspace, one can
find a non-zero vector $(l_0,(l_{\beta})_{\beta\in\Phi})\in\bbf_q^{|\Phi|+1}$ such that, for any $t\in\bbt(\bbf_q)$ one has
\begin{equation}\label{kernel}l_0+\sum_{\beta\in\Phi}l_{\beta}\beta(t)=0.\end{equation}

 On the other hand, since $\bbg$ is simply connected,
$X(\bbt)$ has a base $\{\lambda_{\alpha}\}_{\alpha\in\Delta}$ such
that
$\sigma_{\alpha'}\lambda_{\alpha}=\lambda_{\alpha}-\delta_{\alpha\alpha'}\alpha'$
for any $\alpha$ and $\alpha'$ in $\Delta$, where
$\delta_{\alpha\alpha'}$ is the Kronecker delta. It is easy to see
that any root $\beta$ is equal to
$\sum_{\alpha\in\Delta}\langle\beta,\alpha\rangle\lambda_{\alpha}$.
So these coefficients are entries of the Cartan matrix. In particual, in absolute value, they are less than or equal to $\rho=2$ (resp. $\rho=3$) for any $\beta\in\Phi$ and $\alpha\in\Delta$, when $\bbg$ is not (resp. is) of
type $\G_2$. Hence equation (\ref{kernel}) gives us that, for any $t\in\bbt(\bbf_q)$,
$l_0+\sum_{\beta\in\Phi}l_{\beta}\prod_{\alpha\in\Delta}\lambda_{\alpha}(t)^{\langle\beta,\alpha\rangle}=0.$
Let
$$Q((T_{\alpha})_{\alpha\in\Delta})=(\prod_{\alpha\in\Delta}T_{\alpha})^{\rho}\cdot(l_0+\sum_{\beta\in\Phi}l_{\beta}
\prod_{\alpha\in\Delta}T_{\alpha}^{\langle\beta,\alpha\rangle}).$$ 
By the
above discussion, $Q$ is a polynomial, $\deg_{T_{\alpha}} Q\leq
2\rho$, and $Q((a_{\alpha})_{\alpha\in\Delta})=0$, for any
$(a_{\alpha})_{\alpha\in\Delta}\in(\bbf_q\setminus \{0\})^{|\Delta|}$.
Now lemma~\ref{poly} completes the proof of the lemma.
\end{proof}

By equation (\ref{kstructure}), we know that $\tfr'_k=\hfr_k\cap\Lie(\bbt')$ is
a $k$-structure of $\Lie(\bbt')$. Therefore
$\bbt'=C_{\bbh}(\Lie(\bbt'))=C_{\bbh}(\tfr'_k)$. On the other hand,
since $k$ is an infinite field, one can find $x\in\tfr'_k$ such
that $C_{\Lie(\bbh)}(x)=C_{\Lie(\bbh)}(\tfr'_k)$, and so
$C_{\bbh}(x)=C_{\bbh}(\tfr'_k)$ as
$\Lie(C_{\bbh}(x))=C_{\Lie(\bbh)}(x)$ because $x$ is
semisimple~\cite[chapter III.9]{Bo}. Hence we can see that
for a semisimple element $x\in\Lie(\bbh)(k)$, $\bbt'=C_{\bbh}(x)$
and so it is defined over $k$. Let $t'\in\bbt'(k)$; since
$\Ad(t')(\hfr_k)=\hfr_k$, $\hfr_k$ is a $k$-structure, and
equation (\ref{kstructure}), for any root $\beta'\in\Phi'$, we have $\beta'(t')\in
k$. However $\bbt'(k)$ is Zariski-dense in $\bbt'$ as $k$ is
infinite~\cite[chapter III.8]{Bo}, and so all the roots are
defined over $k$. Therefore $\Ad(\bbt')$ is a $k$-split torus, and
so using~\cite[chapter V.22.6]{Bo} $\bbt'$ is $k$-split. Hence
by~\cite[theorem 18.7, chapter V]{Bo}, $\bbg$ and $\bbh$ are
$k$-split absolutely almost simple groups. As $\phi$ is an
isomorphism between them, it is a $k$-isomorphism, which finishes
this step.


\section{The possible structures of $\Gamma$.}

First we will show that the only point in the building
$X_{\nu}=X(\bbg,k_{\nu})$ which is fixed by $\bbg(\bbf_p)$
is the hyperspecial vertex $\kappa_0$ corresponded to $\bbg(\o_{\nu})$. In order to see
this claim, it is enough to prove the following lemma:

\begin{lem}~\label{fixedpoint}
Let $\bbg$ be a simply connected Chevalley group, and
$p>5$ (resp. $p>7$) a prime number when $\bbg$ is not (resp. is) of type $\G_2$, and $\bbp$ a proper parabolic
$\bbf_q$-subgroup. Then $\bbg(\bbf_p)$ cannot be a subgroup of
$\bbp(\bbf_q)$.

\end{lem}
\begin{proof}
 Using lemma~\ref{root}, one knows that a subspace $V$ of
$\Lie(\bbg)$ which is stable under $\bbt(\bbf_p)$ is a direct sum
of a subspace of $\Lie(\bbt)$ and some of the weight spaces. Now let us proceed by contradiction, and assume that $\bbg(\bbf_p)$ is a subgroup of $\bbp(\bbf_q)$ for some $q$ and an $\bbf_q$-parabolic subgroup $\bbp$. Now let $V=\Lie(\bbp)$. Thus
$V$ is stable under $\Ad(\bbt)(\bbf_p)$ and
$\Ad(\bbg(\bbf_p))$. Hence it is also stable under
$\Ad(\bbg)(\bbf_p)$.  However we know that (see~\cite{Hog}
or~\cite[lemma 4.6]{We}) $\Lie(\Ad(\bbg))$ is irreducible under
$\Ad(\bbg)(\bbf_p)$ unless $\bbg$ is of type $\A_{lp-1},(l\geq 1)$
or $p=2$, and if $\bbg$ is of type $\A_{lp-1}$, then the derived
algebra of $\Lie(\Ad(\bbg))$ is the only
$\Ad(\bbg)(\bbf_p)$-composition factor of $\Lie(\Ad(\bbg))$ which
is a nontrivial $\Ad(\bbg)(\bbf_p)$-module. Hence we can
conclude that $\bbg(\bbf_p)$ cannot be a subgroup of any proper
parabolic subgroup of $\bbg$.
\end{proof}

Note that if $\bbg(\bbf_p)$ fixes a point in the building
different from the hyperspecial vertex $\kappa_0$, it should fix a
vertex different from $\kappa_0$. So it should fixes the ``interval"
between these vertices, and in particular it fixes a neighborhood
of $\kappa_0$. Therefore it would be a subgroup of a proper parahoric
subgroup of $\bbg(\o_{\nu})$. Now we would get a contradiction after
looking at it modulo $\pi_{\nu}$ and using lemma~\ref{fixedpoint}.

As we are working with a lattice of minimum covolume, it is, in particular, a maximal lattice. So let us recall  Rohlfs' maximality criterion. Since we are working with fields of positive characteristic, we refer to proposition 1.4 from~\cite{BP};

\begin{lem}~\label{BP}
Let $k$ be a global function field, $\nu_0$ a place of $k$,
$\bbg$ a simply connected Chevalley group, $\Gamma$ a
maximal lattice in $\bbg(k_{\nu_0})$ which is commensurable with
$\bbg(\o_k(\nu_0))$, and $\Lambda=\Gamma\cap\bbg(k)$. Then the
closure $P_{\nu}$ of $\Lambda$ in $\bbg(k_{\nu})$ is a
parahoric subgroup of $\bbg(k_{\nu})$,
$\Lambda=\bbg(k)\cap\prod_{\nu\neq\nu_0}P_{\nu}$, and $\Gamma$
is the normalizer of $\Lambda$ in $\bbg(k_{\nu_0})$. 
\end{lem}

Going back to the proof of theorem C, so far, we have seen that if $\Gamma$ in $\Omega_1$ minimizes the covolume, then it is commensurable with $\bbg(\o_k(\nu_0))$, where $k$ is a global function field, and $\nu_0$ is one of its places such that $k_{\nu_0}$ is isomorphic to $F$. Moreover, $\bbg(\bbf_p)$ is a subgroup of $\Gamma$. Hence in the setting of lemma~\ref{BP}, $\bbg(\bbf_p)$ is a subgroup of all the parahorics $P_{\nu}$. Thus by the above discussion, for any $\nu\neq\nu_0$, $P_{\nu}=\bbg(\o_{\nu})$, and therefore  $\Lambda=\bbg(\o_k(\nu_0))$ and  $\Gamma=N_{\bbg(k_{\nu_0})}(\bbg(\o_k(\nu_0)))$, which finishes the third step of the proof of theorem C.

\section{$k=\bbf_q(t)$.}
By the previous step, we know that
$\Gamma=N_{\bbg(k_{\nu_0})}(\bbg(\o_k(\nu_0)))$, for a global function
field $k/\bbf_{q_0}$ and a place $\nu_0$, for which $k_{\nu_0}$ is
isomorphic to $\bbf_q((t^{-1}))$. By the first step,
$\bbg(\bbf_q)$ is a subset of $\Gamma$, so it normalizes
$\bbg(k)\cap\bbg(\bbf_q)=\bbg(\bbf_{q_0})$. Hence $q=q_0$ i.e.
$\nu_0$ is a rational place.

In the remaining part of this section, we are going to show that
$k$ is of genus zero. We again proceed by contradiction. First we will give an easy argument for centerless types, and then consider groups with non-trivial center.  

If $\bbg$ is centerless, then it is isomorphic to its adjoint form. Therefore $comm(\bbg(\o_k(\nu_0)))=\bbg(k)$, and so $\Gamma=\Lambda=\bbg(\o_k(\nu_0))$. Now using~\ref{genus}, we get the desired result.

Now assume that $\bbg$ has a non-trivial center. N.~Iwahori and H.~Matsumoto~\cite{IM} showed that there is natural bijection between non-trivial elements of the center of a simply connected Chevalley group $\bbg$, and simple roots with coefficient one in the highest root. On the other hand, a simple root has coefficient one in the highest root if and only if the unipotent radical of the corresponded maximal parabolic is abelian. Hence, we can find a maximal
standard parabolic $\bbp=\bbp_{\alpha}$ whose unipotent radical is
abelian. As we have seen in the proof of the quantitative version
of Kazhdan-Margulis theorem, we may define the automorphism
$\theta=\theta_{\alpha}$, and we would have $\theta(G_1)\subseteq
\bbg(\o_{\nu})$ and $\xi_1(\theta(G_1))\subseteq
R_u(\bbp)(\bbf_q)$. On the other hand,
$\Gamma\cap\bbg(\o_{\nu})=\bbg(\bbf_q)$, and so
$\Gamma\cap\theta(G_1)=R_u(\bbp)(\bbf_q)$. One knows that
$\Gamma\cap\theta(\bbg(\o_{\nu}))$ normalizes
$\Gamma\cap\theta(G_1)$. The following lemmas tell us about the
normalizer of $R_u(\bbg)(\bbf_q)$ in $\bbg(k_{\nu})$:
\begin{lem}~\label{centLevi}
Let $\bbp$ be a standard parabolic whose unipotent radical
$R_u(\bbp)$ is abelian. Then if an element $x$ of
$\bbl=\bbp\cap\bbp^-$ commutes with $R_u(\bbp)(\bbf_q)$, then it
is in the center of $\bbg$.
\end{lem}
\begin{proof}
Let us denote the set of positive roots which ``occur" in
$R_u(\bbp)$ by $\Psi$. First we recall that $\bbu=R_u(\bbp)$,
$\prod_{\beta\in\Psi}\bbu_{\beta}$, and $\bba^{|\Psi|}$ are
isomorphic as $\bbf_q$-varieties. Hence the $\bbf_q$-structure of
the regular functions of $\bbu$ is isomorphic to the ring of
polynomials in $|\Psi|$ variables
$\bbf_q[T_{\beta}]_{\beta\in\Psi}$. Using Steinberg presentation
and the fact that sum of two roots in $\Psi$ is not a root, one can easily see that the vector space generated by
$T_{\beta}$'s is stable under the action of our Levi component
$\bbl$. Therefore for any $y\in\bbl$, there are some constants
$n_{\beta\theta}$ from the universal algebraically closed field such
that
$x^*(T_{\beta})=\sum_{\theta\in\Psi}n_{\beta\theta}T_{\theta}$.
Now if $x$ commutes with the $\bbf_q$ points of $\bbu$, one can
conclude that the matrix $(n_{\beta\theta})$ is identity, and so
$x$ commutes with $\bbu$. Since $\bbu$ is normalized by $\bbl$,
centralizer of $\bbu$ in $\bbl$ is a normal subgroup of $\bbl$.
On the other hand, $\bbl$ is a semisimple algebraic group, and modulo its center is direct product of simple algebraic groups. However none of these factors act trivially on $\bbu$. Therefore $x$ is in the center of $\bbl$.
Since the highest root is an element of $\Psi$, $x$ is also in the center of $\bbg$, as we claimed.
\end{proof}
\begin{lem}~\label{normalizer}
Let $\bbp=\bbp_{\alpha}$ be a standard parabolic whose unipotent
radical is abelian, and $\bbl=\bbp\cap\bbp^-$ be a Levi component;
then:

\begin{itemize}\item[(i)] $N_{\bbg(F)}(R_u(\bbp)(\bbf_q))\subseteq
\bbp(F).$ \item[(ii)]
$N_{\bbg(F)}(R_u(\bbp)(\bbf_q))\cap\bbl(\o)=\bbl(\bbf_q).$
\end{itemize}
\end{lem}
\begin{proof}

\textit{(i)} Let $x\in N=N_{\bbg(F)}(R_u(\bbp)(\bbf_q))$. If $x$
is not in $\bbp(F)$, then by Bruhat decomposition,
$x=pr_{\alpha}u$ for some $p\in\bbp(F)$ and $u\in R_u(\bbp)(F)$,
where $r_{\alpha}$ is a representative of $\alpha$ in
$N_{\bbg}(\bbt)(\bbf_p)$. Therefore, since $R_u(\bbp)$ is abelian
and normalized by $\bbp$, we have:
$$r_{\alpha}R_u(\bbp)(\bbf_q)r_{\alpha}^{-1}\subseteq
R_u(\bbp)(F).$$ However it is impossible as
$r_{\alpha}\bbu_{\alpha}(\bbf_q)r_{\alpha}^{-1}\subseteq
\bbu_{-\alpha}$.

\textit{(ii)} It is clear that $\bbl(\bbf_q)$ is a subset of
$N\cap\bbl(\o)$, so we just need to prove the other side. Let
$h\in N\cap\bbl(\o)$. After multiplying $h$ by an element of
$\bbl(\bbf_q)$, we may and shall assume that $h$ is in the first
congruence subgroup $L_1$ of $\bbl(\o)$. It is easy to see that if
an element of the first congruence subgroup normalizes
$R_u(\bbp)(\bbf_q)$, then it should centralize
$R_u(\bbp)(\bbf_q)$, too. Now by lemma~\ref{centLevi}, we can get
that $h=1$, which completes the proof of the lemma.
\end{proof}

\noindent
 As a corollary of the above lemma, we can see that:
 $$N_{\bbg(k_{\nu})}(R_u(\bbp)(\bbf_q))\cap\bbl(\o)R_u(\bbp)(k_{\nu})=\bbl(\bbf_q)R_u(\bbp)(k_{\nu}).$$
Therefore we have :
$$\begin{array}{ll}
\Gamma\cap\theta(\bbg(\o))&=\Gamma\cap\theta(\bbg(\o))\cap\bbp(k_{\nu})=\Gamma\cap\theta(\bbp(\o))\\ &\\
&=\Gamma\cap\bbl(\o)R_u(\bbp)(\pi^{-1}\o)\\ & \\
&\subseteq \Gamma\cap\bbl(\bbf_q)R_u(\bbp)(\pi^{-1}\o),\hspace{.2in}\mbox{ using lemma~\ref{normalizer}}\\ & \\
&=\bbl(\bbf_q)(\Gamma\cap
R_u(\bbp)(\pi^{-1}\o)),\hspace{.2in}\mbox{since}\hspace{.1in}\bbl(\bbf_q)\subseteq
\Gamma,
\end{array}$$

\noindent where $\pi$ is a uniformizing element of $\o_{\nu}$.
Now, if $u\in \Gamma\cap R_u(\bbp)(\pi^{-1}\o)$, then since
$\Gamma$ normalizes $\bbg(\o_k(\nu))$, for any $t\in
\bbt(\bbf_q)$, $utu^{-1}$ is in $\bbg(\o_k(\nu))$. Hence
$t^{-1}utu^{-1}$ is an element of $R_u(\bbp)(\o_k(\nu))\cap
R_u(\bbp)(\pi^{-1}\o)$. But , by Weierstrass' gap theorem,
$\o_k(\nu)\cap \pi^{-1}\o_{\nu}=\bbf_q$, and so $t^{-1}utu^{-1}$
is in $R_u(\bbp)(\bbf_q)$, so is $u$ when $q>3$. Hence we
have:
$$\Gamma\cap\theta(\bbg(\o))=\bbp(\bbf_q).$$
Therefore by the virtue of the lemma~\ref{index-vol}, we can
conclude that the covolume of $\Gamma$ is bigger than 2, which is
a contradiction.

Hence $k/\bbf_q$ is a genus zero global function field, which has
a place of degree one. Therefore it is isomorphic to the rational
function field $\bbf_q(t)$~\cite[chapter XVI,4]{A}, as we wished.

\section{The final step of the proof of theorem C.}
Let $k=\bbf_q(t)$. Automorphism group of $k$ acts transitively on
the set of rational places, and such automorphism gives rise to
continuous isomorphism between completions of $k$. Hence we may
assume that $\nu$ is the ``infinite" place i.e. $k_{\nu}=
\bbf_q((t^{-1}))$. Let
$$T^+=\{a\in\bbt(k_{\nu})| |\alpha_i(a)|=t^{-n_i}, 0\leq n_i \hspace{.1in}\mbox{for any}\hspace{.1in} 1\leq i\leq r\},$$
where $\Delta=\{\alpha_1,\cdots,\alpha_r\}$ is the set of positive
roots. Using reduction theory from the works of Harder~\cite{Ha1}
or Springer~\cite{Sp}, A.~Prasad~\cite{P} explicitly investigated
the case of Chevalley groups over $\bbf_q(t)$, and he has shown
the following decomposition:
\begin{equation}\label{reduction}\bbg(k_{\nu})=\bbg(\o_{\nu})\cdot T^+\cdot\bbg(\o_k(\nu)).\end{equation}
In fact, he has shown that $T^+$ is a fundamental domain of
$\bbg(\o_k(\nu))=\bbg(\bbf_q[t])$ in $\bbg(k_{\nu})/\bbg(\o_{\nu})$. Now we finish
the proof of theorem A, by establishing the following lemma:
\begin{lem}~\label{normlattice}
Let $\bbg$ be a simply-connected Chevalley group, and $q>3$; then
$$N_{\bbg(\bbf_q((t^{-1})))}(\bbg(\bbf_q[t]))=\bbg(\bbf_q[t]).$$
\end{lem}
\begin{proof}
Let $g\in\Gamma=N_{\bbg(\bbf_q((t^{-1})))}(\bbg(\bbf_q[t]))$. By
equation (\ref{reduction}), there are $x, a,$ and $\gamma$ in
$\bbg(\bbf_q[[t^{-1}]])$, $T^+$ ,and $\Lambda=\bbg(\bbf_q[t])$,
respectively, such that $g=xa\gamma$. Thus $xa\in\Gamma$. If
$a\neq 1$ then for some $\alpha\in \Delta$, $|\alpha(a)|<1$. Thus
$xau_{\alpha}(1)a^{-1}x^{-1}\in \Lambda\cap G_1=1$, which is a
contradiction. Hence $x\Lambda x^{-1}=\Lambda$. As $\bbg(\bbf_q)$
is a subgroup of $\Lambda$, we may assume that $x$ is an element
of the first congruence subgroup, $G_1$. Since intersection of
$\Lambda$ with $G_0=\bbg(\bbf_q([[t^{-1}]])$ is $\bbg(\bbf_q)$ and
$x$ normalizes $G_0$, $x$ should also normalizes $\bbg(\bbf_q)$,
and so it should commute with all elements of $\bbg(\bbf_q)$ since
$x\in G_1$. However we claim that $C_{\bbg}(\bbg(\bbf_q))$ is
central. Since $\bbt(\bbf_q)$ consists of semisimple elements,
$\Lie(C_{\bbg}(\bbt(\bbf_q)))=C_{\Lie\bbg}(\bbt(\bbf_q))$. So for
$q>3$, $C_{\bbg}(\bbt(\bbf_q))=\bbt$, and we can get that
$C_{\bbg}(\bbg(\bbf_q))$ is central, using
$au_{\alpha}(1)a^{-1}=u_{\alpha}(\alpha(a))$. But the first
congruence subgroup intersects the center of $\bbg$ trivially,
which completes the proof of the lemma.
\end{proof}



\noindent
{Dept. of Math., Princeton University, Fine hall, Washington road, Princeton, NJ, 08544, $\&$
IAS, 1 Einstein drive, Princeton, NJ, 08540.}
\\

\noindent
{\em E-mail address:} {\tt asalehi@math.princeton.edu}
\end{document}